\documentclass[12pt]{amsart}
\usepackage{amsmath}
\usepackage{amsthm}
\usepackage{amssymb}
\usepackage{amsfonts,mathrsfs}
\usepackage{stmaryrd}
\usepackage{amsxtra}  
\usepackage{epsfig}
\usepackage{verbatim}

\theoremstyle{plain}
\newtheorem{theorem}[equation]{Theorem}
\newtheorem{proposition}[equation]{Proposition}

\newtheorem{definition}[equation]{Definition}
\theoremstyle{remark}
\newtheorem{remark}[equation]{Remark}
\numberwithin{equation}{section}

\newcommand{\sjump}{\hskip .2 cm}

\newcommand{\ssubset}{\subset\subset}

\begin{document}

\title{Oka's Lemma, convexity, and intermediate positivity conditions}
\author{A.-K. Herbig \& J. D. McNeal}
\dedicatory{Dedicated to J.P. D'Angelo, whose work celebrates positivity in many forms}
\thanks{Research of the first author was supported by Austrian Science Fund FWF grant V187N13. Research of the second author was supported in part by National Science Foundation grant DMS--0752826.}
\address{Department of Mathematics, \newline University of Vienna, Vienna, Austria}
\email{anne-katrin.herbig@univie.ac.at}
\address{Department of Mathematics, \newline Ohio State University, Columbus, Ohio, USA}
\email{mcneal@math.ohio-state.edu}
\begin{abstract}
A new proof of Oka's lemma is given for smoothly bounded, pseudoconvex domains $\Omega\ssubset\mathbb{C}^n$. The method of proof is then also applied to other convexity-like hypotheses on the boundary of $\Omega$.
\end{abstract}
\maketitle

\section{Introduction}

If $\Omega$ is a domain of holomorphy in $\mathbb{C}^n$, Oka's Lemma states that $\phi(z)=-\log d_{b\Omega}(z)$ is plurisubharmonic for $z$ in $\Omega$, where $d_{b\Omega}(z)$ denotes the Euclidean distance from $z$ to $\Omega^c=\mathbb{C}^n\setminus\Omega$. This is a foundational result in several complex variables, with $\phi$ serving as the initial building block in various constructions of holomorphic functions on $\Omega$, e.g., Theorems 4.2.2, 4.4.3, and 4.4.4  in \cite{hormander_scv_book}, Theorem 3.18 in \cite{Range_scv_book}, Theorems~3.4.5 and~5.4.2 in \cite{krantz_scv_book}, and Theorem D.4 in Chapter IX of \cite{GunRos}, among others, hinge on Oka's Lemma.

The aim of this paper is to give a new proof of Oka's Lemma when $\Omega$ has smooth boundary $b\Omega$, and to examine the result as an instance where positivity conditions on the Hessian of a function $f$ are ``spread'' to a wider set of points and vectors by taking functional combinations of $f$ of the form $\chi\circ f$, for $\chi:\mathbb{R}\to\mathbb{R}$.

This point of view is easiest to describe via the signed distance-to-the-boundary function $\delta=\delta_{b\Omega}$; see \eqref{E:sign_dist} below. If $\Omega$ is a smoothly bounded domain of holomorphy, then $\Omega$ is Levi pseudoconvex, see, e.g., Theorem 2.6.12 in
\cite{hormander_scv_book}. Since $\delta$ is a defining function for $\Omega$, it follows that
\begin{equation}\label{E:delta_psc}
\sum_{j,k=1}^n\frac{\partial^2\delta}{\partial z_j\partial\bar z_k}(p)V_j\overline{V}_k\geq 0,\quad\text{if } p\in b\Omega, V\in\mathbb{C}T_p(b\Omega).
\end{equation}
Oka's Lemma says that \eqref{E:delta_psc} implies $\phi=-\log(-\delta)$ is plurisubharmonic on $\Omega$, i.e. that
\begin{equation}\label{E:logdelta_psh}
\sum_{j,k=1}^n\frac{\partial^2\delta}{\partial z_j\partial\bar z_k}(z)W_j\overline{W}_k +\frac 1{d(z)}\left|\sum_{j=1}^n\frac{\partial\delta}{\partial z_j}(z)\, W_j\right|^2 \geq 0,\quad\text{if } z\in\Omega, W\in\mathbb{C}^n.
\end{equation}
Notice that the quadratic form in \eqref{E:delta_psc} is only nonnegative-definite at a small set of points in $\overline\Omega$ (namely, $p\in b\Omega$) and in certain directions (namely, $V\in\mathbb{C}T_p(b\Omega)$), while the form in \eqref{E:logdelta_psh} is nonnegative-definite at all points in $\Omega$ and in all directions. Thus, Oka's Lemma asserts that the positivity (on its complex Hessian) $\phi$ inherits from 
$\delta$ is more widespread than condition \eqref{E:delta_psc} implies at first glance.

This paper grew out of our desire to find a direct proof of Oka's Lemma. The standard proof, see 
Theorems 2.6.12 in \cite{hormander_scv_book}, Theorem~3.3.5 in \cite{krantz_scv_book}, E.5.11 in \cite{Range_scv_book},
is by contradiction: assuming \eqref{E:logdelta_psh} is violated at some $z\in\Omega$ and in some direction $W$, a boundary point $p$ and a direction $V\in\mathbb{C}T_p(b\Omega)$ are found where \eqref{E:delta_psc} cannot hold. The advantage of the canonical approach is the usual one: negation of the non-strict inequalities result in strict inequalities and these are easier to deal with than \eqref{E:delta_psc} and \eqref{E:logdelta_psh} themselves. 

Our proof deals with the semi-definite inequalities \eqref{E:delta_psc} and \eqref{E:logdelta_psh} directly which, we believe, has intrinsic interest. The proof given here re-casts the semi-definite conclusion \eqref{E:logdelta_psh} as another, non-strict inequality on the {\it square} of the distance function, see \eqref{E:main_ineq1}, then uses simple Taylor analysis to show that \eqref{E:delta_psc} implies \eqref{E:main_ineq1}. Variational arguments often fail when one tries to pass from one non-strict inequality to another, so their success in this instance merits mention. 
The local constancy of $\|\nabla\delta\|$ plays a key role in our approach to this issue.

Once Oka's Lemma (Theorem \ref{T:Oka}) is proved in this way, it is illuminating to apply this method to other convexity-like hypotheses on $b\Omega$ besides pseudoconvexity. The most natural hypotheses of this kind are: (i) the real Hessian of $\delta$ non-negative on the real tangent space to $b\Omega$ (convexity), (ii) the real Hessian of $\delta$ non-negative on the complex tangent space to $b\Omega$ ($\mathbb{C}$-convexity), and (iii) the complex Hessian of $\delta$ non-negative on the real tangent space to $b\Omega$ ($\delta$ plurisubharmonic ``on the boundary''). We examine how these hypotheses yield widespread non-negativity on the Hessians of $\delta$ or $-\log(-\delta)$ in Sections
 \ref{S:convex} and  \ref{S:intermediate}. We follow the method used to prove Theorem \ref{T:Oka} quite closely in these sections, in order to clearly identify how the different hypotheses lead to different conclusions. After our paper was written, we learned that \cite{AndPasSig_ccvx_book} earlier gave a proof of the $\mathbb{C}$-convex case along these lines, so our proof of Theorem \ref{T:cconvex} merely reprises their proof. 
 
 In the final part of Section \ref{S:intermediate}, we examine non-negativity of the complex Hessian of $\delta$ on cones of vectors containing the complex tangent space and lying in the real tangent space. Under this hypothesis, we show (Theorem \ref{T:gamma}) how the size of the Diederich-Forn\ae ss exponent (\cite{mcneal05}) --- but only for the fixed defining function $\delta$ --- is determined by the angle of the cone of non-negativity. Theorem \ref{T:gamma} gives a spectrum of results that naturally interpolate between the conclusion given in Theorem \ref{T:Oka} and that given in Theorem \ref{T:psh}. This result explains an example given in \cite{diederichfornaess77b}, where no $\eta >0$ exists such that $-(-\delta)^\eta$ is plurisubharmonic, and is also related to results in \cite{herbigfornaess07}, \cite{herbigfornaess08} which deals with situations where $\eta$ can be chosen close to 1 (but for defining functions other than $\delta$).

\section{Tangent spaces and Hessians}\label{S:TSH}

Succinct notation for Hessians (real and complex) of smooth functions and tangent spaces (real and complex) will make the arguments in Sections  \ref{S:proof} -- \ref{S:intermediate} quite transparent. We present these objects using global coordinates for brevity, mentioning only the invariance needed in the subsequent proofs.

Let $\Omega\subset{\mathbb C}^n$ denote a domain with smooth boundary $b\Omega$. A local defining function for $\Omega$ in a neighborhood $U$ of $p\in b\Omega$, is a real-valued function $r\in C^\infty(U)$ satisfying
$U\cap\Omega=\left\{ z\in U: r(z)<0\right\}$ and $\nabla r(z)\neq 0\text{ for } z\in U$. 

Let $(z_1,\dots ,z_n)$ denote the standard coordinates on $\mathbb{C}^{n}$, with $z_k=x_{2k-1}+i\, x_{2k}$ for $k=1,\dots ,n$. The usual Cauchy--Riemann vector fields are written $$\frac{\partial}{\partial z_k}=\frac 12\left(\frac{\partial}{\partial x_{2k-1}}-i\,\frac{\partial}{\partial x_{2k}}\right)  
\qquad\text{and}\qquad\frac{\partial}{\partial\bar z_k}=\frac 12\left(\frac{\partial}{\partial x_{2k-1}}+i\,\frac{\partial}{\partial x_{2k}}\right)$$ and
differentiation of a smooth function will be denoted with subscripts, e.g., $f_{z_j}=\frac{\partial f}{\partial z_j}$.
The real tangent space to $b\Omega$ at $q\in b\Omega$, ${\mathbb R}T_q(b\Omega)$, is
\begin{equation}\label{E:R_tangent}
{\mathbb R}T_q(b\Omega) =\left\{ W\in{\mathbb R}^{2n} : \sum_{k=1}^{2n} r_{x_k}(q) \, W_k =0\right\}
\end{equation} 
where $W=\sum W_k\frac\partial{\partial x_k}$. Note that if $\left(y_1,\dots, y_{2n}\right)$ is another, smooth coordinate system in a neighborhood of $q$, then $\sum r_{x_k}(q) \, W_k= \sum r_{y_k}(q) \, \widetilde{W}_k$ if 
$W=\sum W_k\frac\partial{\partial x_k}=\sum\widetilde W_k\frac\partial{\partial y_k}$, so \eqref{E:R_tangent} is invariant of coordinate change.
The complex tangent space to $b\Omega$ at $q\in b\Omega$, ${\mathbb C}T_q(b\Omega)$, is 
\begin{equation}\label{E:C_tangent}
{\mathbb C}T_q(b\Omega) =\left\{ V\in{\mathbb C}^n : \sum_{k=1}^n r_{z_k}(q) \, V_k =0\right\},
\end{equation} 
where $V=\sum V_k\frac\partial{\partial z_k}$. If $(w_1,\dots, w_n)$ is an arbitrary local holomorphic coordinate system near $q$, the vector fields $\frac\partial{\partial w_k}$ are given by the chain rule, and $V=\sum \widetilde{V}_k\frac\partial{\partial w_k}$ is decomposed with respect to the frame $\left\{\partial/\partial w_1,\dots, \partial/\partial w_n\right\}$, it is easy to see \eqref{E:C_tangent} is an invariant definition.
Both \eqref{E:R_tangent} and  \eqref{E:C_tangent} are independent of the choice of local defining function for $\Omega$. 

The Hessian of a smooth function $f: {\mathbb C}^n\longrightarrow {\mathbb C}$ can be viewed as a bilinear form on vectors in $\mathbb{R}^{2n}$  or on vectors in $\mathbb{C}^n$. We invert the usual presentation by considering its action on complex vectors first. The {\it real Hessian} of $f$ at a point $p$ acting on the pair of vectors $(A,B)\in\mathbb{C}^n\oplus \mathbb{C}^n$ is 
\begin{equation}\label{E:H_complex}
{\mathcal H}_{f(p)}\big(A,B\big) = 2\,\text{Re }\left(\sum_{k,\ell=1}^n f_{z_k z_\ell}(p) A_kB_\ell\right) +
2\,\sum_{k,\ell=1}^n f_{z_k \bar z_\ell}(p) A_k\overline{B}_\ell.
\end{equation}
Checking that \eqref{E:H_complex} agrees with the more familiar definition of the Hessian using the underlying real coordinates requires a small computation.  We first fix a specific identification of $\mathbb{C}^n$ and $\mathbb{R}^{2n}$; 
if  $A=(a_1+i\, a_2,\dots,a_{2n-1}+i\, a_{2n})$ and $B=(b_1+i\, b_2,\dots, b_{2n-1}+i\, b_{2n})$
are in $\mathbb{C}^{n}$, then the corresponding vectors in $\mathbb{R}^{2n}$, $(a_1, a_2,\dots,a_{2n-1},a_{2n})$ and
 $(b_1, b_2,\dots,b_{2n-1},b_{2n})$, will also be denoted by the symbols $A$ and $B$.
The definition of the operators $\frac{\partial}{\partial z_k}, \frac{\partial}{\partial\bar z_k}$ and straightforward linear algebra shows
\begin{equation*}
{\mathcal H}_{f(p)}\big(A,B\big) = \sum_{k,\ell=1}^{2n}\frac{\partial^2 f}{\partial x_k\partial x_\ell}(p)a_kb_\ell.
\end{equation*}
The {\it complex Hessian} of $f$ at $p$ is (one-half of) the second term on the right-hand side of \eqref{E:H_complex}:
\begin{equation}\label{E:L_complex}
{\mathcal L}_{f(p)}\big( A,B\big) =\sum_{k,\ell=1}^n f_{z_k \bar z_\ell}(p) A_k\overline{B}_\ell.
\end{equation}
One-half the first term on the right-hand side of \eqref{E:H_complex} --- henceforth, the {\it complement of the Levi form} --- will be denoted
\begin{equation}\label{E:Q_complex}
{\mathcal Q}_{f(p)}\big( A,B\big) = \text{Re }\left(\sum_{k,\ell=1}^n f_{z_k z_\ell}(p) A_kB_\ell\right).
\end{equation}

The forms $\mathcal{L}$ and $\mathcal{Q}$ transform differently under multiplication of their arguments by $i=\sqrt{-1}$. 
Indeed, \eqref{E:L_complex} immediately yields ${\mathcal L}_{f(p)}\big( A,B\big)={\mathcal L}_{f(p)}\big( iA,iB\big)$, while \eqref{E:Q_complex}
shows ${\mathcal Q}_{f(p)}\big( A,B\big)= -{\mathcal Q}_{f(p)}\big(i A,i B\big)$. Consequently, for all pairs 
$(A,B)\in\mathbb{C}^n\oplus \mathbb{C}^n$
\begin{equation*}
{\mathcal H}_{f(p)}\big(A,B\big) + {\mathcal H}_{f(p)}\big(iA,iB\big)= 4{\mathcal L}_{f(p)}\big( A,B\big).
\end{equation*}

It is also convenient to have notation for first-derivative expressions of $f$. The complex gradient of $f$ acting on a vector in $\mathbb{C}T \left({\mathbb C}^n\right)$ will be denoted
\begin{equation}\label{E:complex_gradient}
\left<\partial f(p),V\right> =\sum_{k=1}^n f_{z_k}(p) V_k,
\end{equation}
when $V=\sum V_k \frac\partial{\partial z_k}$. The symbol $\left<\bar\partial f(p),V\right> $ is defined analogously. The real gradient of $f$ acting on a vector $W\in\mathbb{R}^{2n}$ will be denoted
\begin{equation}\label{E:real_gradient}
\left\langle \nabla f(p),W\right\rangle=\sum_{k=1}^{2n} f_{x_k}W_k,
\end{equation}
if $W=\sum_{k=1}^{2n} W_k\frac\partial{\partial x_k}$. 
\medskip

If $f\in C^3(U)$, $U$ open in $\mathbb{C}^n$, and $p=(p_1+ip_2,\dots, p_{2n-1}+i p_{2n}), q=(q_1+iq_2,\dots, q_{2n-1}+i q_{2n})$ are two points in $U$, Taylor's theorem to second-order in real notation says
\begin{equation*}
f(q)=f(p)+\left<\nabla f(p),V\right> +\frac{1}{2} \mathcal{H}_{f(p)}\left(V,V\right)+\mathcal{O}\left(\left\|V\right\|^3\right),
\end{equation*}
where $V=\left(p_1-q_1,\dots ,p_{2n}-q_{2n}\right)\in\mathbb{R}^{2n}$. In complex notation, the same result is expressed
\begin{align}\label{E:complex_taylor}
f(q)=f(p)+ 2\,\text{Re}\left<\partial f(p),W\right> + \mathcal{Q}_{f(p)}\left(W,W\right)  &+\mathcal{L}_{f(p)}\left(W,W\right) \\ 
&+\mathcal{O}\left(\left\|W\right\|^3\right),\notag
\end{align}
where $W= p-q\in\mathbb{C}^n$.
\medskip

Basic convexity notions, on both functions and domains, are easily expressed using the above notation.

\begin{definition}\label{D:convex_psh_fns}
Let $U \subset\mathbb{C}^n$ be an open set and $f\in C^2(U)$. Then

(a) $f$ is convex at $p\in U$ if
\begin{align*}
\mathcal{H}_{f(q)}\big(W,W\big)\geq 0\qquad \forall\sjump q\in U', \sjump W\in\mathbb{C}^n
\end{align*}
for some neighborhood $U'\subset U$ containing $p$.
\smallskip

(b) $f$ is plurisubharmonic at $p\in U$ if 
\begin{align*}
\mathcal{L}_{f(q)}\big(W,W\big)\geq 0\qquad \forall\sjump q\in U', \sjump W\in\mathbb{C}^n
\end{align*}
for some neighborhood $U'\subset U$ containing $p$.
\end{definition}

If $f$ is convex at p, then so is $f\circ L$ for any $\mathbb{R}$-affine coordinate change of the standard coordinates. This follows easily from the chain rule. Plurisubharmonicity is not invariant under a general $\mathbb{R}$-affine coordinate change. But it is invariant under an arbitrary, local biholomorphic map (again, by the chain rule), in particular under a $\mathbb{C}$-affine coordinate change.

\begin{definition}\label{D:convex_psc_domains}
Let $\Omega\subset\mathbb{C}^n$ be a smoothly bounded open set, $p_0\in b\Omega$,
and $r$ is a local defining function for $\Omega$ in a neighborhood of $p_{0}$. Then

(a)  $\Omega$ is convex near $p_{0}$ if

\begin{align*}
\mathcal{H}_{r(p)}\big(V,V\big)\geq 0\qquad \forall\sjump p\in U\cap b\Omega, \sjump V\in\mathbb{R}T_p(b\Omega)
\end{align*}
for some neighborhood $U$ containing $p_{0}$.

(b)  $\Omega$ is pseudoconvex near $p_{0}$ if

\begin{align*}
\mathcal{L}_{r(p)}\big(V,V\big)\geq 0\qquad \forall\sjump p\in U\cap b\Omega, \sjump V\in\mathbb{C}T_p(b\Omega)
\end{align*}
for some neighborhood $U$ containing $p_{0}$.
\end{definition}

Both conditions in Definition \ref{D:convex_psc_domains} are independent of the choice of local defining function. The
conditions are also invariant under a $\mathbb{C}$-affine coordinate change, as mentioned above.

\medskip

\section{Distance to the boundary}\label{S:dist}

The other ingredient in Oka's lemma is the distance-to-the-boundary function, which we denote by $d= d_{b\Omega}$:
$$d(z)=\inf_{q\in b\Omega}\|z-q\|.$$
The signed distance to $b\Omega$ will be denoted by $\delta=\delta_{b\Omega}$:
\begin{align}\label{E:sign_dist}
\delta(z)=\left\{\aligned-d(z),&\qquad z\in\Omega \\ d(z),&\qquad z\in \Omega^c\endaligned\right. \;.
\end{align}

We collect some basic facts about $\delta$ on a smoothly bounded domain in $\mathbb{C}^n$. 
 
 \begin{proposition}\label{P:distance}
If $\Omega\subset\mathbb{C}^{n}$ is a smoothly bounded domain, then there exists a neighborhood $U$ of $b\Omega$ such that:
\begin{itemize}
  \item[(a)] The map $b_{b\Omega}: U\longrightarrow b\Omega$ satisfying $\|b_{b\Omega}(z)-z\|=
  |\delta_{b\Omega}(z)|$  is well-defined.
  
  \item[(b)] The functions $b_{b\Omega}$ and $\delta_{b\Omega}$ are smooth on $U$. 
 
\item[(c)] For each $p\in b\Omega$, let $\nu_{p}$ be the real outward unit normal to $b\Omega$ at $p$. Then there exists a coordinate system $\left(w_1,\dots, w_n\right)$, $w_k=y_{2k-1}+iy_{2k}, k=1,\dots, n$,
which is a $\mathbb{C}$-affine coordinate change of the standard coordinates on $\mathbb{C}^{n}$, such that for all $q=t\nu_p\in U\cap\Omega$, $t\in\mathbb{R}$,
\begin{align}
  \delta_{y_j}(q)=
  \begin{cases}
    0, & j\neq 2n-1\\
    1, & j=2n-1
  \end{cases}\;\;.
  \end{align}
\end{itemize}
\end{proposition}
For a  proof of (a) see, e.g., \cite{federer59}, Lemma 4.1.1 on pgs. 444--445. Proofs of (b) and (c) follow from Corollary 5.2 in \cite{HerMcN09} after using Lemma 1, pg. 382, in \cite{GilbargTrudinger}.

\section{A proof of Oka's Lemma}\label{S:proof}

The significant content of Oka's Lemma is that $-\log d_{b\Omega}$ is plurisubharmonic \textit{near} $b\Omega$, if $\Omega$ is pseudoconvex. Since the new feature in our proof also occurs near $b\Omega$, we shall focus on proving the following 

\begin{theorem}[Version of Oka's Lemma]\label{T:Oka} Let $\Omega$ be a smoothly bounded, pseudoconvex domain in $\mathbb{C}^n$. There exists a neighborhood $U$ of $b\Omega$ such that $-\log\left(-\delta(z)\right)$ is plurisubharmonic for $z\in U\cap\Omega$.

\end{theorem}

\begin{proof} For the expansion of a normed expression below (see \eqref{E:D_main}), it is convenient to consider the square of the function $d$ rather than $d$ (or $\delta$) itself; let
$$D(z)=d_{b\Omega}^2(z)=\inf \left\{ \| z-q\|^2: q\in b\Omega\right\}.$$
Obviously, $-\log (-\delta (z))$ is plurisubharmonic iff $-2\log(- \delta (z))$ is plurisubharmonic,
and $-2\log(-\delta (z)) = -\log D(z)$ if $z\in\Omega$.  Thus, it suffices to show there exists a neighborhood $U$ of $b\Omega$ such that
\begin{align}\label{E:main_ineq1}
{\mathcal L}_{D(z)}\big(V,V\big)\leq \frac {\left|\left<\partial D(z), V\right>\right|^2}{D(z)}
\end{align}
for all $z\in U\cap\Omega$ and all $V\in{\mathbb C}^n$.

\medskip

Let $U$ be a small enough neighborhood of $b\Omega$ so that the projection map $b$ is well-defined and smooth.
For a given $q\in U\cap\Omega$, make the $\mathbb{C}$-affine coordinate change in Proposition \ref{P:distance} to achieve
\begin{itemize}
\item[(i)] $b(q) = 0$   $\left(=(0,\dots, 0) = (0',0)\in {\mathbb C}^{n-1}\times{\mathbb C})\right)$
\item[(ii)] $q=\left(0', a\right),\quad a<0$
\item[(iii)] ${\mathbb R}T_0\left(b\Omega\right) =\left\{ \text{Re}z_n=0\right\}$.
\end{itemize}
We shall continue to denote the changed coordinates as $(z_1,\dots, z_n)$, with $z_k= x_{2k-1}+ i x_{2k}$, 
and do all subsequent computations with respect to these coordinates.

In a neighborhood $U_q$ of 0, the Implicit Function Theorem says that $b\Omega$ can be viewed as a smooth graph over ${\mathbb R}T_0\left(b\Omega\right)$. Explicitly, we can find a defining function, $r(z)$, of the form
\begin{align}\label{E:implicit_defining}
r(z)=\text{Re } z_n +h\left(z', \text{Im }z_n\right),
\end{align}
where $h\in C^2(U_{q})$, $h(0)=0$ and $\nabla h(0) = (0,\dots, 0)$. 

Clearly $D(q)=a^2$. It follows from Proposition \ref{P:distance} that  $D_{x_{2n-1}}(q)= 2a$ and that all the other real partial derivatives of $D$ vanish at $q$. This translates to the following information on the complex partials of $D$:
\begin{align}\label{E:D_complex_partials}
D_{z_k}(q)=\left\{\aligned 0&\qquad\text{if } k\neq n, \\ a&\qquad\text{if } k=n.\endaligned\right.
\end{align}

Let $V=\left(V', V_n\right)$ denote an arbitrary direction in $\mathbb{C}^n$, with $\|V\|$ small enough so that $q+V$ lies in $U_q$. Decompose $V_n$ into its real and imaginary parts, $V_n = s+it$, and note
$$q +V =\left(V', a +s+ it\right).$$
The form of the defining function $r$ suggests a suitable point on $b\Omega$ with which to estimate $D(q+V)$:  \eqref{E:implicit_defining} says that $\left(V', -h(V',0)\right)\in b\Omega$ for any $(V',0)\in U_q$. Consequently,
\begin{align}\label{E:D_main}
D(q+V)&\leq\|(V', a +s+ it) -(V', -h(V',0))\|^2 \\
&=\| a+V_n +h(V',0)\|^2\notag \\ 
&= a^2 +2(a+s)\cdot h(V',0)+ 2a s +|V_n|^2+ h^2(V',0)\notag
\\ &=a^2+2a\cdot h(V',0)+ 2a s +|V_n|^2+ {\mathcal O}\left(\|V\|^3\right)\notag.
\end{align}
The last equality follows since $h$ vanishes to second order at 0.

Set $\widetilde{V}=(V',0)$ and notice that $h(\widetilde{V})=r(\widetilde{V})$. Since $\widetilde{V}\in{\mathbb R}T_0(b\Omega)$, Taylor's theorem gives
\begin{align}\label{E:Taylorh}
h(\widetilde{V})= \frac{1}{2}\,{\mathcal H}_{r(0)}\big(\widetilde{V}, \widetilde{V}\big) +{\mathcal O}\Bigl(\bigl\|\widetilde{V}\bigr\|^3\Bigr).
\end{align}
However, $\widetilde{V}$ actually belongs to ${\mathbb C}T_0(b\Omega)$, so 
${\mathcal L}_{r(0)}\big(\widetilde{V},\widetilde{V}\big)\geq 0$ by pseudoconvexity. Since $a <0$, it follows that
the second-order part of $2a\cdot h(V',0)$ in \eqref{E:D_main} corresponding to the complex Hessian is negligible, i.e., that
\begin{align*}
a{\mathcal H}_{r(0)}\big(\widetilde{V}, \widetilde{V}\big)\leq a\, 2{\mathcal Q}_{r(0)}\big(\widetilde{V},\widetilde{V}\big).
\end{align*}
Returning to \eqref{E:D_main}, we obtain the estimate
\begin{align*}
D(q+V)&\leq a^2 + 2a\cdot{\mathcal Q}_{r(0)}\big(\widetilde{V}, \widetilde{V}\big)+ 2a s +|V_n|^2 + {\mathcal O}\left(\|V\|^3\right)\notag \\ &= D(q) + 2a\cdot{\mathcal Q}_{r(0)}\big(\widetilde{V}, \widetilde{V}\big)+  2\text{Re}\left(\left<\partial D(q),V\right>\right) +
\frac{\left|\left<\partial D(q),V\right>\right|^2}{D(q)} \\ &\hspace{8.8cm}+{\mathcal O}\left(\|V\|^3\right),\notag
\end{align*}
where \eqref{E:D_complex_partials} and the fact that $D(q)=a^2$ are used to obtain the last equality. A similar estimate holds
in the direction $iV$, the only changes occurring in the second and third terms:
\begin{align*}
D(q+iV)&\leq D(q) - 2a\cdot{\mathcal Q}_{r(0)}\big(\widetilde{V}, \widetilde{V}\big)-  2\text{Im}\left(\left\langle\partial D(q),V\right\rangle\right) +
\frac{\left|\left<\partial D(q),V\right>\right|^2}{D(q)} \\ &\hspace{8.9cm}+{\mathcal O}\left(\|V\|^3\right).\notag
\end{align*}
Adding these two estimates yields, for $S=D(q+V)+D(q+iV)$,
\begin{align}\label{E:D2}
S\leq 2D(q)+F(q,V) +
2\,\frac{\left|\left<\partial D(q),V\right>\right|^2}{D(q)} +{\mathcal O}\left(\|V\|^3\right),
\end{align}
where $F(q,V)= 2\,\text{Re}\left(\left<\partial D(q),V\right>\right)  -2\,\text{Im}\left(\left<\partial D(q),V\right>\right)$. 
\bigskip

On the other hand, expanding $D(q+V)$ and $D(q+iV)$ about $q$ by Taylor's theorem gives
\begin{align*}
D(q+V)&= D(q) + 2\text{Re}\left(\left<\partial D(q),V\right>\right) +  {\mathcal Q}_{D(q)}(V,V)+{\mathcal L}_{D(q)}
(V,V) \\&\hspace{8.5cm}+ {\mathcal O}\left(\|V\|^3\right)
\end{align*}
and
\begin{align*}
D(q+iV)&= D(q) - 2\text{Im}\left(\left<\partial D(q),V\right>\right) - {\mathcal Q}_{D(q)}(V,V)+{\mathcal L}_{D(q)}
(V, V)\\ &\hspace{8.5cm}+ {\mathcal O}\left(\|V\|^3\right).
\end{align*}
Adding these two equations yields
\begin{align}\label{E:D3}
S= 2D(q)+F(q,V) + 2\, {\mathcal L}_{D(q)}
(V, V)+ {\mathcal O}\left(\|V\|^3\right).
\end{align}
Estimating \eqref{E:D3} from above by \eqref{E:D2} and making the obvious cancellations yields
\begin{align*}
2\, {\mathcal L}_{D(q)}(V, V)\leq 2\,\frac{\left|\left<\partial D(q),V\right>\right|^2}{D(q)} +{\mathcal O}\left(\|V\|^3\right).
\end{align*}
Homogeneity considerations in $V$ then show \eqref{E:main_ineq1} holds for $q\in U_q\cap\Omega$. Since the argument above can be given for every $q\in U\cap\Omega$, the proof is complete.
\end{proof}

\section{Convex domains}\label{S:convex}

If $\Omega\subset\mathbb{C}^n$ is convex, it is not necessary to compose $\delta$ with a function like $\chi(x)=-\log(-x)$ in order to get a conclusion related to Theorem \ref{T:Oka}. Indeed, if $\mathcal{H}_{\delta(p)} \geq 0$ on ${\mathbb R}T_p(b\Omega)$ for $p\in b\Omega$, then $\delta$ itself inherits widespread positivity on its real Hessian:
 
 \begin{theorem}\label{T:convex}
 Let $\Omega$ be a smoothly bounded, convex domain in $\mathbb{C}^n$. There exists a neighborhood $U$ of $b\Omega$ such that 
 $\delta(z)$ is a convex function for $z\in U\cap\Omega$.
 \end{theorem} 
 
 Different proofs of Theorem \ref{T:convex} are known, see pgs. 354--357 in \cite{GilbargTrudinger}, pgs. 57--60 in \cite{hormander_convex_book} and Corollaries 5.7 and 5.12 in \cite{HerMcN09}. In fact, $\delta$ is convex on a full neighborhood of $b\Omega$, not just on $U\cap\Omega$; see Remark \ref{R:convexoutside} below. As mentioned in the introduction, the proof below is parallel to the proof of Theorem \ref{T:Oka}, to clearly trace how the stronger hypothesis in Theorem \ref{T:convex} leads to its stronger conclusion.

 \begin{proof}[Proof of Theorem \ref{T:convex}]
 As before, consider the function $D(z)=\left(\delta(z)\right)^{2}$.  A straightforward computation gives
\begin{align}\label{E:HessDHessdelta}
\mathcal{H}_{D(z)}\big(V,V\big)=\frac{\left|\langle\nabla D(z), V\rangle \right|^{2}}{2D(z)}+2\delta(z) 
\mathcal{H}_ {\delta(z)}\big(V,V\big)
\end{align}
for all $z$ near $b\Omega$ and $V\in\mathbb{C}^{n}$. To prove Theorem \ref{T:convex}, it therefore suffices to show that there is a neighborhood $U$ of $b\Omega$ such that
\begin{align}\label{E:main_ineq2}
{\mathcal H}_{D(z)}\big(V,V\big)\leq \frac {\left|\left<\nabla D(z), V\right>\right|^2}{2\,D(z)}\qquad\sjump\forall\sjump z\in U\cap\Omega,\sjump V\in\mathbb{C}^{n}.
\end{align}

Let $U$ be a small enough neighborhood of $b\Omega$ so that the projection map $b$ is well-defined and smooth.
Fix $q\in U\cap\Omega$, make the $\mathbb{C}$-affine coordinate change in Proposition \ref{P:distance} and obtain
\begin{itemize}
\item[(i)] $b (q) = 0$   $\left(=(0,\dots, 0) = (0',0)\in {\mathbb C}^{n-1}\times{\mathbb C})\right)$
\item[(ii)] $q=\left(0', a\right),\quad a <0$
\item[(iii)] ${\mathbb R}T_0\left(b\Omega\right) =\left\{ \text{Re}z_n=0\right\}$.
\end{itemize}
Continue to denote the changed coordinates as $(z_1,\dots, z_n)$, with $z_k= x_{2k-1}+ i x_{2k}$, as in the proof of Theorem \ref{T:Oka}.
 
 Apply the Implicit Function Theorem as before: there exists a neighborhood $U_q$ of the origin, a function $h\in C^\infty(U_{q})$ with
 $h(0)=0$ and $\nabla h(0)=(0,\dots 0)$, such that
 \begin{align}\label{E:implicit_defining2}
r(z)=\text{Re } z_n +h\left(z', \text{Im }z_n\right)
\end{align}
is a local defining function for $\Omega$ in $U_{q}$.
      
 Clearly $D(q)=a^2$, while $D_{x_{2n-1}}(q)= 2a$ and all the other partial derivatives of $D$ vanish at $q$, by Proposition \ref{P:distance}.
 Let $V=(V',V_{n})\in\mathbb{C}^{n}$ be given, write $V_n=s+i t$, and consider $q+V$ as a small perturbation of $q$. We have that 
 $\big(V',-h(V',c)+ ic\big)$ lies in $b\Omega$, for any $c\in\mathbb{R}$. Thus,
 \begin{align*}
 D(q+V)&\leq\left\|(V',a+V_{n})-\left(V',-h(V',c)+ic\right)\right\|^{2} \\
 &=\left\|a+s +h(V',c)+i(t-c)\right\|^{2}\\
 &= \left(a+s+h(V', t)\right)^2,
 \end{align*}
 if $c$ is chosen equal to $t$. Expanding this square yields
 \begin{align}\label{E:D_main2}
 D(q+V)&\leq a^2 + 2as +s^2+2(a+s)\cdot h(V',t) + h^2(V',t)
\\ &=a^2 + 2a s +\left(\text{Re}\,V_n\right)^2 +2a\cdot h(V',t)+ {\mathcal O}\left(\|V\|^3\right),\notag
\end{align}
since $h$ vanishes to second order at 0.
 
 Now set $\widetilde{V}=(V',it)$. Note that $h(V',t)=r(\widetilde{V})$ and that $\widetilde{V}\in\mathbb{R}T_0(b\Omega)$ (though not in $\mathbb{C}T_0(b\Omega)$, unless $t=0$). Taylor's theorem gives
 \begin{align}\label{E:Taylorh2}
 h(V',t)= \frac 12\,{\mathcal H}_{r(0)}\big(\widetilde{V}, \widetilde{V}\big) +{\mathcal O}\Bigl(\bigl\|\widetilde{V}\bigr\|^3\Bigr).
 \end{align}
 Convexity of $\Omega$ implies $H_{r(0)}(\widetilde{V},\widetilde{V})$ is 
 non-negative, so we have $h(V',t)\geq -C\|\widetilde{V}\|^{3}$ for some constant $C>0$. Because $a<0$, it follows from
 \eqref{E:D_main2} that
   \begin{align}\label{E:D2a}
     D(q+V)&\leq a^2+2a s +\left(\text{Re}\,V_n\right)^2+ {\mathcal O}\Bigl(\bigr\|V\bigl\|^3\Bigr) \\
    &= D(q)+\bigl\langle\nabla D(q),V\bigr\rangle+
     \frac{|\langle\nabla D(q),V\rangle|^{2}}{4D(q)}+\mathcal{O}\Bigl(\bigl\|V\bigr\|^{3}\Bigr).\notag
   \end{align}
   
  However, expanding $D(q+V)$ about $q$ by Taylor's theorem gives
   \begin{align}\label{E:D3a}
     D(q+V)=D(q)+\langle\nabla D(q),V\rangle+\frac{1}{2} \mathcal{H}_{D(q)}(V,V)+
     \mathcal{O}\Bigl(\bigl\|V\bigr\|^{3}\Bigr).
   \end{align}
   Estimating \eqref{E:D3a} from above by \eqref{E:D2a} leads to
   \begin{align*}
     \frac 12\,\mathcal{H}_{D(q)}(V,V)\leq\frac{|\langle\nabla D(q),V\rangle|^{2}}{4D(q)}+\mathcal{O}\Bigl(\bigl\|V\bigr\|^{3}\Bigr).
   \end{align*}
   The homogeneity in $V$ then implies that \eqref{E:main_ineq2} holds.
   \end{proof}
   
   \begin{remark}\label{R:convexoutside} 
   Under the hypothesis of Theorem \ref{T:convex}, $\delta(z)$ is also convex for $z\in U\cap\Omega^c$.
   The same initial part of the proof above is used; however  $a>0$ when $q\in U\cap\Omega^c$.
   Notice that the conclusion above \eqref{E:D2a}
   can be improved: Taylor's theorem actually yields
   \begin{align*}
     h(V',t)=r(0)+\langle\nabla r(0),\widetilde{V}\rangle
     +\frac 12\,\mathcal{H}_{r(\alpha)}\left(\widetilde{V},\widetilde{V}\right)
   \end{align*}
  for some point $\alpha$ on the line segment connecting the origin and the point $\widetilde{V}$. It follows from the proof of Proposition 4.1 in \cite{HerMcN09} that $\mathcal{H}_{r(\alpha)}(\widetilde{V},\widetilde{V})\geq 0$. Therefore, 
   $h(V',t)$ is non-negative, and the distance of the point $q+V$ to $b\Omega$ is larger or equal to its distance to the 
  hyperplane $\{x\in\mathbb{R}^{n} : x_{n}=0\}$. But the latter is attained at the point $\widetilde{V}$.  It follows that
  \begin{align*}
    D(q+V)&\geq\|(V',a+V_{n})-\widetilde{V}\|^{2}=\|a+s\|^{2}\\
    &=D(q)+\langle\nabla D(q),V\rangle+\frac{|\langle\nabla D(q),V\rangle|^{2}}{4D(q)}.
  \end{align*}
  This yields, by repeating the arguments in the proof of Theorem \ref{T:convex}, that $\delta$ is convex on $\Omega^{c}\cap U$. 
  \end{remark}
 \section{Intermediate positivity conditions}\label{S:intermediate}
 
 In the previous two sections, hypotheses on the Hessians and tangent spaces were ``matched'' with respect to the real or complex structure: $\mathcal{H}_\delta \geq 0$ on 
 ${\mathbb R}T(b\Omega)$ in Theorem \ref{T:convex} and $\mathcal{L}_\delta \geq 0$ on ${\mathbb C}T(b\Omega)$ in Theorem \ref{T:Oka}. In this section, we study ``mixed'' situations.
 
\subsection*{Non-negativity  of $\mathcal{H}_{\delta(p)}$ on $\mathbb{C}T_{p}(b\Omega)$}

\begin{definition}\label{D:cconvex}
Let $\Omega\subset\mathbb{C}^n$ be a smoothly bounded open set, $p_0\in b\Omega$,
and $r$ is a local defining function for $\Omega$ in a neighborhood of $p_{0}$. Then
$\Omega$ is $\mathbb{C}$-convex near $p_{0}$ if
\begin{align*}
\mathcal{H}_{r(p)}\big(V,V\big)\geq 0\qquad \forall\sjump p\in U\cap b\Omega, \sjump V\in\mathbb{C}T_p(b\Omega)
\end{align*}
for some neighborhood $U$ containing $p_{0}$.
\end{definition}

As with the conditions in Definition \ref{D:convex_psc_domains}, $\mathbb{C}$-convexity is independent of the choice of local defining function as well as invariant under a $\mathbb{C}$-affine coordinate change.

A convex domain is clearly $\mathbb{C}$-convex , since ${\mathbb C}T(b\Omega)\subset {\mathbb R}T(b\Omega)$. Also, the displayed equation below \eqref{E:Q_complex} shows that a $\mathbb{C}$-convex domain is pseudoconvex. Not suprisingly, a result intermediate to Theorems \ref{T:Oka} and \ref{T:convex} holds for $\mathbb{C}$-convex domains.

 \begin{theorem}\label{T:cconvex}
 Let $\Omega$ be a smoothly bounded, $\mathbb{C}$-convex domain in $\mathbb{C}^n$. There exists a neighborhood $U$ of $b\Omega$ such 
 that 
 \begin{align}\label{E:ccvxOka}
   \mathcal{H}_{\delta(z)}(V,V)\geq\frac{\left|\langle\nabla\delta(z),iV\rangle \right|^{2}}{\delta(z)}
   \sjump\qquad\forall\sjump z\in U\cap\Omega, \sjump V\in\mathbb{C}^{n}.
 \end{align}
\end{theorem} 
As mentioned in the introduction, Theorem \ref{T:cconvex} is proved in \cite{AndPasSig_ccvx_book}, see the implication (ii) to (iii) of Theorem 2.5.18 therein (there is a notational difference between our paper and \cite{AndPasSig_ccvx_book} -- compare \eqref{E:H_complex} and the first displayed equation on pg.~60 in \cite{AndPasSig_ccvx_book}). 
\begin{proof}[Proof of Theorem \ref{T:cconvex}]
  As in the proofs of Theorems \ref{T:Oka} and \ref{T:convex} work with the function $D(z)=(\delta(z))^{2}$. Note first that for any vector $V\in\mathbb{C}^{n}$
  \begin{align*}
    \left|\langle\nabla\delta(z),V\rangle \right|^{2}
    =
    \frac{\left|\langle\nabla D(z),V\rangle \right|^{2}}{4D(z)}.
  \end{align*}
  It then follows from \eqref{E:HessDHessdelta} that \eqref{E:ccvxOka} is equivalent to
  \begin{align}\label{E:ccvxD}
    \mathcal{H}_{D(z)}(V,V)\leq 
    \frac{\left|\langle\nabla D(z),V\rangle \right|^{2}}{2D(z)}+\frac{\left|\langle\nabla D(z),iV\rangle \right|^{2}}{2D(z)}
   \qquad\sjump\forall\sjump z\in U\cap\Omega,\sjump V\in\mathbb{C}^{n}.
  \end{align}
  To prove \eqref{E:ccvxD}, proceed as in the proof of Theorem \ref{T:Oka}, starting below \eqref{E:main_ineq1}.
  Take $q=(0',a)$ for $a<0$ and $V=(V',V_{n})$ with $V_{n}=s+it$. Choose again the boundary point  $(V',-h(V',0))$ to obtain an upper bound on $D(q+V)$:
    \begin{align*}
    D(q+V)&\leq a^2+2a\cdot h(V',0)+ 2a s +s^{2}+t^{2}+ {\mathcal O}\left(\|V\|^3\right).
  \end{align*}
  Since $(V',0)$ is in $\mathbb{C}T_{0}(b\Omega)$, it follows from \eqref{E:Taylorh} and the hypothesis of $\mathbb{C}$-convexity that $h(V',0)\geq-C\|V\|^{3}$ for some $C>0$. Therefore
  \begin{align*}
    D(q+V)&\leq a^{2}+ 2a s +s^{2}+t^{2}+ {\mathcal O}\left(\|V\|^3\right)\\
    &=D(q)+\left\langle\nabla D(q),V\right\rangle+
    \frac{\left|\langle\nabla D(z),V\rangle \right|^{2}}{4D(z)}+\frac{\left|\langle\nabla D(z),iV\rangle \right|^{2}}{4D(z)}  + {\mathcal O}\left(\|V\|^3\right).
  \end{align*}
  Using \eqref{E:D3a}, it then follows that
  \begin{align*}
    \frac{1}{2} H_{D(q)}(V,V)\leq 
    \frac{\left|\langle\nabla D(z),V\rangle \right|^{2}}{4D(z)}+\frac{\left|\langle\nabla D(z),iV\rangle \right|^{2}}{4D(z)} 
    + {\mathcal O}\left(\|V\|^3\right),
  \end{align*}
  which implies \eqref{E:ccvxD}.
\end{proof}

\medskip
\subsection*{Non-negativity  of $\mathcal{L}_{\delta(p)}$ on $\mathbb{R}T_{p}(b\Omega)$} 

Finally, we turn to the case of non-negativity of the \textit{complex} Hessian of a defining function on the \textit{real} tangent space.
Unlike the previous conditions, this one is not independent of the choice of defining function.  We shall only consider this condition on $\delta$ as our method of proof is fine-tuned to this defining function.

It is elementary that this positivity spreads to arbitrary directions:

\begin{remark}\label{R:rotation}
 If $\mathcal{L}_{\delta(z)}(V,V)\geq 0$ for all $z\in b\Omega$ and $V\in\mathbb{R}T_{z}(b\Omega)$, then
  $\mathcal{L}_{\delta(z)}(W,W)\geq 0$ for all $z\in b\Omega$ and $W\in\mathbb{C}^{n}$. This can be seen if, e.g., the coordinates in the proof of Theorem \ref{T:convex} are used. Then, for $W\in\mathbb{C}^{n}$, choose $\theta\in[0,2\pi)$ such that $\text{Re}(e^{i\theta}W_{n})=0$. The complex Hessian of $\delta$ is invariant under such rotations and
  $V:=e^{i\theta}W\in\mathbb{R}T_{0}(b\Omega)$.
\end{remark}

Moreover, this positivity spreads off $b\Omega$:

\begin{theorem}\label{T:psh}
  Let $\Omega$ be a smoothly bounded domain in $\mathbb{C}^{n}$. Suppose that $\mathcal{L}_{\delta(z)}(V,V)\geq 0$ for all $z\in b\Omega$ and $V\in\mathbb{R}T_{z}(b\Omega)$. Then there exists a neighborhood $U$ of $b\Omega$ such that $\delta$ is plurisubharmonic on $U\cap\Omega$, i.e., $\mathcal{L}_{\delta}(z)(V,V)\geq 0$ for all $z\in U\cap \Omega$ and $V\in\mathbb{C}^{n}$.
\end{theorem}

\begin{proof}
  Again,  use the function $D(z)=(\delta(z))^{2}$, and note that
  \begin{align}\label{E:complexHessD}
    \mathcal{L}_{D(z)}(W,W)=\frac{\left|\langle\partial D(z),W \rangle\right|^2}{2D(z)}+2\delta(z)\mathcal{L}_{\delta(z)}(W,W)
  \end{align}
  for $z$ near $b\Omega$ and $W\in\mathbb{C}^n$. Thus, to prove Theorem \ref{T:psh} it suffices to show that there exists a neighborhood $U$ of $b\Omega$ such that
  \begin{align}\label{E:Dpshclaim}
     \mathcal{L}_{D(z)}(W,W)\leq\frac{\left|\langle\partial D(z),W\rangle \right|^2}{2D(z)}\qquad\sjump\forall\sjump z\in U\cap\Omega,\sjump W\in\mathbb{C}^{n}.
  \end{align}
  
  Proceed as in the proof of Theorem \ref{T:convex}, starting below \eqref{E:main_ineq2}. 
   By Remark \ref{R:rotation}, it suffices to
   prove  \eqref{E:Dpshclaim} for vectors that are of the form $V=(V',it)$ for $V'\in\mathbb{C}^{n-1}$ and $t\in\mathbb{R}$.  As in the proof of 
   Theorem \ref{T:convex} , take $q=(0',a)$ for $a<0$, and choose the boundary point
   $(V',-h(V',t)+it)$ to obtain an upper bound for $D(q+V)$:
    \begin{align}\label{E:D(q+V)psh} 
      D(q+V)&\leq \left\|(V',a+V_{n})-(V',-h(V',t)+it) \right\|^{2} \notag \\
       &= D(q)+2a\cdot h(V',t)
    +\mathcal{O}\Bigl(\bigl\|V\bigr\|^{3} \Bigr).
   \end{align}
   
   Next, choose the boundary point $(iV',-h(iV',0))$ to obtain an upper bound for $D(q+iV)$:
  \begin{align}\label{E:D(q+iV)psh}
    D(q+iV)
    &\leq \left\|(iV',a+iV_{n})-(iV',-h(iV',0)) \right\|^{2}
    \notag\\
    &= D(q)-2\text{Im}\left(\langle\partial D(q),V\rangle\right)+2a\cdot h(iV',0)+
    \frac{\left|\langle\partial D(q),V \rangle\right|^{2}}{D(q)}
    +\mathcal{O}\Bigl(\bigl\|V\bigr\|^{3} \Bigr).
  \end{align}
  Suppose (temporarily) that there exists some constant $C>0$ such that
  \begin{align}\label{E:intermediatepshclaim}
    h(V',t)+h(iV',0)\geq -C\|V\|^{3}.
  \end{align}
  Then,  adding \eqref{E:D(q+iV)psh} to \eqref{E:D(q+V)psh}  would yield
  \begin{align*}
       S&=D(q+V)+D(q+iV)\\
      & \leq 2D(q)-2 \text{Im}\left(\langle\partial D(q),V\rangle\right)
       +  \frac{\left|\langle\partial D(q),V \rangle\right|^{2}}{D(q)} +\mathcal{O}\Bigl(\bigl\|V\bigr\|^{3} \Bigr).
     \end{align*}
     Using \eqref{E:D3} for $S$ and the fact that $\text{Re}\left(\langle\partial D(q),V\rangle\right)=0$, it then would follow that
     \begin{align*}
      2\mathcal{L}_{D(q)}(V,V)\leq  \frac{\left|\langle\partial D(q),V \rangle\right|^{2}}{D(q)} +\mathcal{O}\Bigl(\bigl\|V\bigr\|^{3} \Bigr),
     \end{align*}
     which implies \eqref{E:Dpshclaim}. 
     
     Thus it remains to show that \eqref{E:intermediatepshclaim} holds. 
     For that write $\widehat{V}=(iV',0)$, so that $h(iV',0)=r(\widehat{V})$ (and $h(V',t)=r(V)$). Then
     Taylor's Theorem gives
     \begin{align*}
       h(V',t)+h(iV',0)=\frac{1}{2}\left(\mathcal{H}_{r(0)}\bigl(V,V\bigr)+\mathcal{H}_{r(0)}\bigl(\widehat{V},\widehat{V}\bigr) \right)
       +\mathcal{O}\left(\|V\|^{3} \right).
     \end{align*}
  Since $\delta_{x_{j}x_{k}}(0)=r_{x_{j}x_{k}}(0)$  (see for instance part (i) of Remark 4.2 in \cite{HerMcN09} with $r=\delta$ there), it suffices to show that
  \begin{align*}
    \mathcal{H}_{\delta(0)}(V,V)+\mathcal{H}_{\delta(0)}(\widehat{V},\widehat{V})\geq 0.
 \end{align*}
 As $\widehat{V}=iV+(0',t)$, the bilinearity of $\mathcal{H}$ yields
  \begin{align*}
    \mathcal{H}_{\delta(0)}(\widehat{V},\widehat{V})=
    \mathcal{H}_{\delta(0)}(iV,iV)
    +2\mathcal{H}_{\delta(0)}\bigl(iV,(0',t)\bigr)
    +\mathcal{H}_{\delta(0)}\bigl((0',t),(0',t) \bigr).
  \end{align*}
  However, the fact that $\sum_{j=1}^{2n}\left|\delta_{x_{j}}\right|^{2}=1$ holds in a neighborhood of $b\Omega$ implies that
 $$\mathcal{H}_{\delta(0)}\bigl(\,.\,,(0',1)\bigr)=0,$$
   see, e.g.,  (5.5) in \cite{HerMcN09} for details.
   Thus  $\mathcal{H}_{\delta(0)}(\widehat{V},\widehat{V})=\mathcal{H}_{\delta(0)}(iV,iV)$,
  so that by  \eqref{E:H_complex}
  \begin{align*}
    \mathcal{H}_{\delta(0)}(V,V)+\mathcal{H}_{\delta(0)}(\widehat{V},\widehat{V})
    =
    4\mathcal{L}_{\delta(0)}\bigl(V,V \bigr)
   \geq 0,
  \end{align*}
  since the complex  Hessian of $\delta$ on the boundary is non-negative definite on the real tangent space.
  \end{proof}
  
  \medskip
  
\subsection*{Non-negativity  of $\mathcal{L}_{\delta(p)}$ on cones in $\mathbb{R}T_{p}(b\Omega)$}

 One may also consider non-negativity of the real or complex Hessian of $\delta$ on cones of vectors, contained in the real tangent space, whose axes are the complex tangent space. We shall only consider the non-negativity of $\mathcal{L}_{\delta}$ on such cones here, but mention $\mathcal{H}_{\delta}\geq 0$ could be considered as well and a result analogous to Theorem \ref{T:gamma} below obtained for that hypothesis.
  
  \begin{definition}
   Let $\Omega\subset\mathbb{C}^{n}$ be a smoothly bounded open set, $p\in b\Omega$. Let  $r$ be a  defining function for $\Omega$ in a neighborhood of $p$ and $\gamma\in[0,\infty)$. Then 
   \begin{align*}
     \mathbb{R}T_{p}^{\gamma}(b\Omega)=\left\{V\in\mathbb{R}T_{p}(b\Omega): 
     \frac{\left|\langle i\nabla r(p),V\rangle \right|}{\|\nabla r(p)\|}
     \leq\gamma
     \left\|
     V- \frac{\langle i\nabla r(p),V\rangle i\nabla r(p)}{\|\nabla r(p)\|^{2}}
     \right\|
     \right\}
   \end{align*}
  \end{definition}
  
 Note that the definition of the cone $\mathbb{R}T_{p}^{\gamma}(b\Omega)$ is independent of the choice of defining function. Also, 
  $\mathbb{R}T_{p}^{\gamma}(b\Omega)$ is invariant under $\mathbb{C}$-affine coordinate changes that are compositions of translations and rotations. Furthermore,  
\begin{itemize}
\item[(i)]  $\mathbb{R}T_{p}^{0}(b\Omega)$ equals $\mathbb{C}T_{p}(b\Omega)$,
\item[(ii)] $\lim_{\gamma\to\infty}\mathbb{R}T_{p}^{\gamma}(b\Omega)$ equals $\mathbb{R}T_{p}(b\Omega)$.
\end{itemize}

  \begin{definition}
    Let $\Omega\subset\mathbb{C}^{n}$ be a smoothly bounded open set, $p_{0}\in b\Omega$, and $U$ a neighborhood of $p_{0}$. Then
    $\delta$ is $\gamma$-plurisubharmonic on $U\cap b\Omega$ if
    \begin{align*}
      \mathcal{\mathcal{L}}_{\delta(p)}(V,V)\geq 0\qquad\sjump\forall\sjump p\in U\cap b\Omega,\sjump V\in\mathbb{R}T_{p}^{\gamma}(b\Omega).
    \end{align*}
  \end{definition}
  Note that the condition of $\gamma$-plurisubharmonicity  of $\delta$ is a condition intermediate to pseudoconvexity of $\Omega$ ($\gamma=0$) and plurisubharmonicity of $\delta$ on $b\Omega$ ($\gamma=\infty$). The following theorem establishes that the complex Hessian of $\delta$ then inherits non-negativity intermediate to the results of Theorem \ref{T:Oka} and Theorem \ref{T:psh}.
  
    \begin{theorem}\label{T:gamma}
    Let $\Omega\subset\mathbb{C}^{n}$ be a smoothly bounded domain. Suppose $\delta$ is  $\gamma$-plurisubharmonic on $b\Omega$ for some $\gamma>0$. 
    Let $\eta=1-2/(2+\gamma^{2})$. Then there exists a neighborhood $U$ of $b\Omega$ such that 
    \begin{align*}
      \mathcal{L}_{-(-\delta)^{\eta}(z)}(V,V)\geq 0\qquad\sjump\forall\sjump z\in U\cap\Omega,
      \sjump V\in\mathbb{C}^{n}.
    \end{align*}
  \end{theorem}
  \begin{proof}
   Use the function $D(z)=\left(\delta(z)\right)^{2}$. First note 
    \begin{align*}
      \mathcal{L}_{-(-\delta)^{\eta}(z)}(V,V)&=\eta(-\delta(z))^{\eta-2}\Bigl((-\delta(z))
      \mathcal{L}_{\delta(z)}(V,V)+(1-\eta)\left|\langle\partial\delta(z),V\rangle \right|^{2}
      \Bigr)\\
      &=
      \eta(-\delta(z))^{\eta-2}\Bigl((-\delta(z))
      \mathcal{L}_{\delta(z)}(V,V)+(1-\eta)\frac{\left|\langle\partial D(z),V\rangle \right|^{2}}{4D(z)}
      \Bigr).
    \end{align*}
    It follows from \eqref{E:complexHessD} that it suffices to show that there exists a neighborhood $U$ of $b\Omega$ such that
    \begin{align}\label{E:claimDDF}
      \mathcal{L}_{D(z)}(V,V)&\leq (2-\eta)\frac{\left|\langle\partial D(z),V\rangle \right|^{2}}{2D(z)}\notag\\
            &=\left(1+\frac{2}{2+\gamma^{2}}\right)\frac{\left|\langle\partial D(z),V\rangle \right|^{2}}{2D(z)}
    \end{align}
    for all $z\in U\cap\Omega$ and $V\in\mathbb{C}^{n}$.
    
    Proceed as in the proof of  Theorem \ref{T:convex}, starting below \eqref{E:main_ineq2}. In particular, let $q=(0',a)$ for $a<0$. 
    By the arguments in Remark \ref{R:rotation}, it suffices to prove \eqref{E:claimDDF} at $z=q$ for vectors $V=(V',V_{n})$ with
    $V_{n}=it$ for $t\in\mathbb{R}_{0}^{+}$. 
    
    Let us first suppose that $V\in\mathbb{R}T_{0}^{\gamma}(b\Omega)$. Then, since $\delta$ is $\gamma$-plurisubharmonic on $b\Omega$, the proof of Theorem \ref{T:psh} is applicable so that 
    \eqref{E:Dpshclaim} holds (which implies \eqref{E:claimDDF} for any $\gamma\geq 0$).
    
    Next, suppose that $V\notin\mathbb{R}T_{0}^{\gamma}(b\Omega)$, i.e., $|t|>\gamma\|V'\|$. Since $(W',-h(W',c)+ic)$ is a boundary point for any
     point $(W',ic)$ sufficiently close to the origin, it follows that
     \begin{align}\label{E:DqVDF}
       D(q+V)&\leq
       \left\|
       (V', a+V_{n})-(W',-h(W',c)+ic)
       \right\|^{2}\notag\\
       &=(a+h(W',c))^{2}+\left\|V'-W'\right\|^{2}+(t-c)^{2}.
     \end{align}
     Similarly, since $(iW',-h(iW',0))$ is a boundary point, it follows that
     \begin{align}\label{E:DqiVDF}
       D(q+iV)&\leq\left\|(iV',a+iV_{n})-\left(iW',-h(iW',0)\right) \right\|^{2}\notag\\
       &=
       \left(a-t+h(iW',0)\right)^{2}+\|V'-W'\|^{2}.
     \end{align}
     Adding \eqref{E:DqiVDF} to \eqref{E:DqVDF} gives
     \begin{align*}
       S&=D(q+V)+D(q+iV)\\
       &\leq
       (a+h(W',c))^{2}+ \left(a-t+h(iW',0)\right)^{2}+2\left\|V'-W'\right\|^{2}+(t-c)^{2}.
     \end{align*}
     Minimizing the function $f(W',c)=2\|V'-W'\|^{2}+(t-c)^{2}$ subject to the constraint   $|c|=\gamma\|W'\|$
     (so that $(W',ic)\in\mathbb{R}T_{0}^{\gamma}(b\Omega)$ holds), yields the minimal value
     \begin{align*}
       f(W_{0}',c_{0})=\frac{2t^{2}}{2+\gamma^{2}}\left(1-\frac{\|V'\|\gamma}{t} \right)^{2},
     \end{align*}
     where $c_{0}=\gamma(t\gamma+2\|V'\|)/(2+\gamma^{2})$ and $W'_{0}=c_{0}V'/(\gamma\|V'\|)$ if $V'\neq 0$. If $V'=0$, we take $W_0=0$.
     Since $V\notin \mathbb{R}T_{0}^{\gamma}(b\Omega)$ it follows that $ f(W_{0}',c_{0})\leq 2t^{2}/(2+\gamma^{2})$, and hence
     \begin{align}\label{E:SDF}
       S\leq &2D(q)
       -2 \text{Im}\left(\langle\partial D(q),V\rangle\right)\\
       &+2a\bigl(h(W_{0}',c_{0})+h(iW_{0}',0)\bigr)+\left(1+\frac{2}{2+\gamma^{2}}\right)\frac{|\langle\partial D(q), V\rangle|^{2}}{D(q)}+\mathcal{O}\left(\|V\|^{3} \right),\notag
     \end{align}
     where it was used that $\|(W_{0}',c_{0})\|=\mathcal{O}(\|V\|)$.
     
     Suppose (temporarily) that there exists a constant $C>0$ such that
     \begin{align}\label{E:intermediategammapshclaim}
       h(W_{0}',c_{0})+h(iW_{0}',0)\geq -C\left\|(W_{0}',c_{0})\right\|^{3}.
     \end{align}
     Then using  \eqref{E:D3} for $S$ in \eqref{E:SDF} and the fact that $\text{Re}\left(\langle\partial D(q),V\rangle\right)=0$, would yield
     \begin{align*}
       2\mathcal{L}_{D(q)}(V,V)\leq \left(1+\frac{2}{2+\gamma^{2}}\right)\frac{\left|\langle\partial D(q), V \rangle \right|^{2}}{D(q)}+
       \mathcal{O}\left(\|V\|^{3} \right),
     \end{align*}
     which would imply \eqref{E:claimDDF}.
     
     That \eqref{E:intermediategammapshclaim} is indeed true may be shown by arguments analogous to the ones in the proof of \eqref{E:intermediatepshclaim}, using the facts that $\delta$ is $\gamma$-plurisubharmonic on $b\Omega$ and that $(W_{0}',ic_{0})$ was chosen to be in $\mathbb{R}T_{0}^{\gamma}(b\Omega)$, the cone of non-negativity of the complex Hessian of $\delta$.
  \end{proof}
  
 \begin{remark}
 In \cite{diederichfornaess77b}, pages 134--137, an example of a pseudoconvex domain is given such that $-(-\delta)^\eta$ is not plurisubharmonic for \textit{any} $\eta >0$. It is straightforward to check for this example, using the computations in \cite{diederichfornaess77b}, that $\mathcal{L}_{\delta(0)}(V,V)<0$ for any $V\in\mathbb{R}T_0\setminus
 \mathbb{C}T_0$, i.e., that $\delta$ is \textit{not} $\gamma$-plurisubharmonic for \textit{any} $\gamma>0$.
 
 \end{remark}

\bibliographystyle{plain}
\bibliography{HerMcN11}
\end{document}